\documentclass[a4paper,11pt]{article}

\usepackage[top=3.0cm, bottom=3.0cm, inner=3.0cm, outer=3.0cm, includefoot]{geometry}

\usepackage{verbatim}

\usepackage{geometry}
\usepackage{amssymb}
\usepackage{amsmath}
\usepackage{graphicx}
\usepackage{amsthm}
\usepackage{bbm}

\usepackage{color}

\usepackage[T1]{fontenc}
\usepackage[utf8]{inputenc}
\usepackage{authblk}

\usepackage{enumerate}
\usepackage{tikz}
\usetikzlibrary{arrows}
\usetikzlibrary{decorations.markings}

\newcommand{\eChar}{\begin{enumerate}[(i)]}
\newcommand{\eCharR}{\begin{enumerate}[(a)]}
\newcommand{\eBr}{\begin{enumerate}[(1)]}

\tikzstyle myBG=[line width=3pt,opacity=1.0]

\newcommand{\drawLinewithBG}[3]
{
	\draw[white,myBG]  (#1) -- (#2);
	\draw[#3,very thick] (#1) -- (#2);
}

\newcommand{\drawlabelnode}[3]
{
	\draw[fill] (#1) circle [radius=0.08];
	\node [#2] at (#1) {\Large #3};
}

\newcommand{\drawHEX}[1]
{
	\draw[black,thick,shift={(#1)}] (-3,0)--(-2,0)--(-2,1)--(-1,1)--(0,0)--(1,0)--(1,1)--(2,1)--(3,0)--(4,0)--(4,1)--(5,1)--(6,0)--(7,0);
	\draw[black,thick,shift={(#1)}] (-4,2)--(-3,2)--(-2,1)--(-1,1)--(-1,2)--(0,2)--(1,1)--(2,1)--(2,2)--(3,2)--(4,1)--(5,1)--(5,2)--(6,2);
}

\title
{Long scale Ollivier-Ricci curvature of graphs}

\author[1]{D. Cushing}
\author[1]{S. Kamtue}
\affil[1]{Department of Mathematical Sciences, Durham University}

\date{\today}

\theoremstyle{plain}
\newtheorem{lemma}{Lemma}[section]
\newtheorem{theorem}[lemma]{Theorem}
\newtheorem{proposition}[lemma]{Proposition}
\newtheorem{corollary}[lemma]{Corollary}

\theoremstyle{definition}

\newtheorem{remark}[lemma]{Remark}
\newtheorem{definition}{Definition}[section]

\newtheorem{example}[lemma]{Example}

\newtheorem{problem}[lemma]{Problem}

\numberwithin{equation}{section}

\begin{document}

\maketitle

\begin{abstract}
  We study the long scale Ollivier-Ricci curvature of graphs as a function of the chosen idleness. As in the previous work on the short scale, we show that this idleness function is concave and
  piecewise linear with at most $3$ linear parts. We provide bounds on the length of the first and last linear pieces. We also study the long scale curvature inside the Cartesian product
  of two regular graphs.
\end{abstract}

\section{Introduction and statement of results}

Ricci curvature is a fundamental notion in the study of Riemannian manifolds. This notion
has been generalised in various ways from the smooth setting of manifolds to more general
metric  spaces. For example, in \cite{Oll} Ollivier introduced a notion of Ricci curvature on metric spaces (later known as ``Ollivier-Ricci curvature''). This gives rise to a notion of Ricci curvature on graphs taking values on the edges and based on optimal transport of lazy random walks, with respect to an idleness parameter $p$. In \cite{LLY11} this notion was modified on graphs to give the ``Lin-Lu-Yau'' curvature. 

Beyond recent theoretical work on this notion (see \cite{BCLMP17, BM15, CKLLS2017, CLMP2018, LLY13, OllVil}), there have been several applications outside mathematics such as in biology (see \cite{Far, SGR, WM17}) and in computing (see \cite{Ni, WJBan, WJB} ).

In \cite{BCLMP17} the authors investigate the Ollivier-Ricci idleness function $p\mapsto \kappa_p(x,y)$, which takes the idleness parameter $p\in[0,1]$ as a variable and gives the value of curvature between the fixed two adjacent vertices $x$ and $y$ (or equivalently, the curvature given on an edge of the graph joining $x$ and $y$). They prove their main theorem that the idleness function $p\mapsto \kappa_p(x,y)$ is concave and piecewise linear over $[0,1]$ with at most 3 linear parts, and it is linear on the intervals $$\left[0,\frac{1}{\textup{lcm}(d_x,d_y)+1}\right] \textup{ and } \left[\frac{1}{\textup{max}(d_x,d_y)+1} ,1\right]$$  In this paper, we do similar investigation on the idleness function, but the condition that the two vertices are adjacent is replaced by distance $\ge 2$ apart (and henceforth called ``long scale curvature'' as in contrast to ``short scale curvature''). Our main result is that the (long scale) idleness function $p\mapsto \kappa_p(x,y)$ is concave and piecewise linear over $[0,1]$ with at most 3 linear parts, and it is linear on the intervals $$\left[0,\frac{1}{\textup{lcm}(d_x,d_y)+1}\right] \textup{ and } \left[\frac{1}{2}-\frac{1}{2}\cdot\frac{d_x+d_y}{2d_xd_y-d_x-d_y},1\right].$$
This main result is split into two theorems, which are stated and proved in Section \ref{sect: 3 linear} and \ref{sect: critical points}. In Section \ref{sect: example}, we provide an example of a graph that has exactly 3 linear parts and the first and the last linear parts are the same intervals as mentioned in the main result. In Section \ref{sect: cartesian}, we give the formula of the long scale curvature of Cartesian products of regular graphs. In Section \ref{sect:  behavior}, we present some interesting behaviours of long scale curvature, including the hexagonal tiling, discrete Bonnet-Myers' theorem, and one of the open problems suggested by Ollivier.
\\
\\
Throughout this article, let $G=(V,E)$ be a locally finite graph with
vertex set $V$, edge set $E$, and which contains no multiple edges or
self loops. Let $d_x$ denote the degree of the vertex $x\in V$ and 
$d(x,y)$ denote the length of the shortest path between two vertices
$x$ and $y$, that is, the combinatorial distance. We denote the existence of an edge between $x$ and $y$ by $x\sim y.$

We define the following probability measures $\mu_x$ for any
$x\in V,\: p\in[0,1]$:
$$\mu_x^p(z)=\begin{cases}p,&\text{if $z = x$,}\\
\frac{1-p}{d_x},&\text{if $z\sim x$,}\\
0,& \mbox{otherwise.}\end{cases}$$
Let $W_{1}$ denote the 1-Wasserstein distance between two probability measures on $V$, see \cite{Vill03} page 211.  The $ p-$Ollivier-Ricci curvature between $x$ and $y$ in $G=(V,E)$ is
$$\kappa_{ p}(x,y)=1-\frac{W_1(\mu^{ p}_x,\mu^{ p}_y)}{d(x,y)}.$$
Y. Lin, L. Lu, and S.T. Yau introduced in \cite{LLY11} the following
Ollivier-Ricci curvature:
$$\kappa_{LLY}(x,y) =\lim\limits_{ p\rightarrow 1}\frac{\kappa_{ p}(x,y)}{1- p}.$$
In particular, we call the curvature $\kappa_p(x,y)$ ``short scale'' when $x\sim y$, and ``long scale'' when $d(x,y)\geq 2$.


\section{Definitions and notation}\label{sect:defn-nots}
We now introduce the relevant definitions and notation we will need in this paper. First, we recall the Wasserstein distance and
the Ollivier-Ricci curvature.

\begin{definition}
Let $G = (V,E)$ be a locally finite graph. Let $\mu_{1},\mu_{2}$ be two probability measures on $V$. The {\it Wasserstein distance} $W_1(\mu_{1},\mu_{2})$ between $\mu_{1}$ and $\mu_{2}$ is defined as
\begin{equation} \label{eq:W1def}
W_1(\mu_{1},\mu_{2})=\inf_{\pi \in \Pi(\mu_1,\mu_2)} \sum_{y\in V}\sum_{x\in V} d(x,y)\pi(x,y),
\end{equation}
where 
\[
\Pi(\mu_1,\mu_2) = \left\{ \pi: V \times V \to [0,1] : \mu_{1}(x)=\sum_{y\in V}\pi(x,y), \;
\mu_{2}(y)=\sum_{x\in V}\pi(x,y) 
\right\}.
\]
\end{definition}

The transport plan $\pi$ moves a mass
distribution given by $\mu_1$ into a mass distribution given by
$\mu_2$, and $W_1(\mu_1,\mu_2)$ is a measure for the minimal effort
which is required for such a transition.
If $\pi$ attains the infimum in \eqref{eq:W1def} we call it an {\it
  optimal transport plan} transporting $\mu_{1}$ to $\mu_{2}$.

\begin{definition}
The $ p-$Ollivier-Ricci curvature of two vertices $x$ and $y$ in $G=(V,E)$ is
$$\kappa_{ p}(x,y)=1-\frac{W_1(\mu^{ p}_x,\mu^{ p}_y)}{d(x,y)} ,$$
where $p$ is called the {\it idleness}.
\end{definition}

A fundamental concept in optimal transport theory and vital to our work is Kantorovich duality. First we recall the notion
of 1--Lipschitz functions and then state the Kantorovich Duality Theorem.

\begin{definition}
Let $G=(V,E)$ be a locally finite graph, $\phi:V\rightarrow\mathbb{R}.$ We say that $\phi$ is $1$-Lipschitz if
$$|\phi(x) - \phi(y)| \leq d(x,y)$$
for all $x,y\in V.$ Let \textrm{1--Lip} denote the set of all $1$--Lipschitz functions on $V$.
\end{definition}

\begin{theorem}[Kantorovich duality \cite{Vill03}]\label{Kantorovich}
Let $G = (V,E)$ be a locally finite graph. Let $\mu_{1},\mu_{2}$ be two probability measures on $V$. Then
$$W_1(\mu_{1},\mu_{2})=\sup_{\substack{\phi:V\rightarrow \mathbb{R}\\ \phi\in \textrm{\rm{1}--{\rm Lip}}}}  \sum_{x\in V}\phi(x)(\mu_{1}(x)-\mu_{2}(x)).$$
If $\phi \in \textrm{\rm{1}--{\rm Lip}}$ attains the supremum we call it an \emph{optimal Kantorovich potential} transporting $\mu_{1}$ to $\mu_{2}$.
\end{theorem}

Most of the time, it is sufficient to consider 1-Lipschitz functions that only yield integer values. This observation is an important tool to deal with Kantorovich potential as in the following way.

\begin{definition}
Let $G=(V,E)$ be a locally finite graph and let $\phi:V\rightarrow\mathbb{R}$. Define two functions $\lfloor\phi\rfloor, \lceil\phi\rceil: V\rightarrow\mathbb{Z}$ to be $\lfloor\phi\rfloor(v):= \lfloor \phi(v) \rceil$ and $\lceil\phi\rceil(v):= \lceil \phi(v) \rceil$.
\end{definition}

\begin{lemma}\textup{(\cite[Lemma 3.2]{BCLMP17})}
\label{lemma:floorlip}
Let $G=(V,E)$ be a locally finite graph. Suppose that a function $\phi:V\rightarrow \mathbb{R}$ is 1-Lipschitz. Then the functions $\lfloor\phi\rfloor$ and $\lceil\phi\rceil$ are also 1-Lipschitz.
\end{lemma}

The existence of integer-valued optimal Kantorovich potentials can be formulated as in the following proposition, which generalises the result from \cite[Lemma 3.3]{BCLMP17}.

\begin{proposition}[Integer-Valuedness] \label{prop:int_valuedness}
Let $G=(V,E)$ be a locally finite graph. Let $\mu_1,\mu_2$ be two probability measures on $V$. Then there exists an integer-valued optimal Kantorovich potential $\phi:V\rightarrow \mathbb{Z}$ transporting $\mu_1$ to $\mu_2$, that is
\begin{align} \label{eqn:OKP}
W_1(\mu_{1},\mu_{2})=  \sum_{x\in V}\phi(x)(\mu_{1}(x)-\mu_{2}(x)).
\end{align}
\end{proposition}

\begin{proof}
Let a 1-Lipschitz function $\phi^*:V\rightarrow \mathbb{R}$ be an optimal Kantorovich potential transporting $\mu_1$ to $\mu_2$, that is
$$W_1(\mu_{1},\mu_{2})=  \sum_{x\in V}\phi^*(x)(\mu_{1}(x)-\mu_{2}(x)).$$
By Proposition \ref{lemma:floorlip}, the function $\lfloor \phi^* \rfloor \in \textup{1-Lip}$. We will show that $\phi=\lfloor \phi^* \rfloor$ satisfies \eqref{eqn:OKP}, and it is therefore an integer-valued optimal Kantorovich potential.

Let $\pi$ be an optimal transport plan transporting $\mu_1$ to $\mu_2$. Construct a graph $H$ with vertices $V$ and edges given by its adjacency matrix $A_H$:
$$A_H(v,w)=\begin{cases}
1 &\textup{ if }\pi(v,w)>0 \textup{ or } \pi(w,v)>0 \\
0 &\textup{ otherwise.}
\end{cases}$$

For each $v\in V$, define $\delta_v:=\phi^*(v)-\lfloor\phi^*(v)\rfloor \in [0,1)$ the fractional part of $\phi^*(v)$.

For any $v,w\in V$ such that $v\stackrel{H}{\sim} w$ (that is $\pi(v,w)>0$ or $\pi(w,v)>0$), complementary slackness theorem gives $|\phi^*(v)-\phi^*(w)|=d(v,w)$. Therefore,
\begin{align*}
|\lfloor\phi^*\rfloor(v)-\lfloor\phi^*\rfloor(w)|
=|\phi^*(v)-\delta_v-\phi^*(w)+\delta_w|
=|\pm d(v,w)-(\delta_v-\delta_w)|.
\end{align*}
which implies that $\delta_v-\delta_w$ has an integer value. Since $\delta_v-\delta_w\in (-1,1)$, so it must be $0$. In conclusion, $\delta_v=\delta_w$ for all $v\stackrel{H}{\sim} w$. By transitivity, $\delta_v=\delta_w$ for all $v,w$ in the same connected component of $H$.

Now let $(W_i)_{i=1}^n$ denote the connected components of $H$. For each $i$, set $\delta_i:=\delta_u$ for any $u\in W_i$. Note also that $\sum\limits_{v\in W_i}\mu_1(v)=\sum\limits_{v\in W_i}\mu_2(v)$ for all $i$, since no mass is transported between different connected components of $H$. Therefore,
\begin{align*}
\sum_{v\in V}\lfloor\phi^*\rfloor(v)(\mu_{1}(v)-\mu_{2}(v))
&=\sum_{v\in V}\phi^*(v)(\mu_{1}(v)-\mu_{2}(v))-
\sum\limits_{v\in V} \delta_v(\mu_{1}(v)-\mu_{2}(v)) \\
&=W_1(\mu_1,\mu_2)-
\sum\limits_{i=1}^{n} \sum\limits_{v\in W_1} \delta_v(\mu_{1}(v)-\mu_{2}(v))\\
&=W_1(\mu_1,\mu_2)-
\sum\limits_{i=1}^{n} \delta_i \Big(\sum\limits_{v\in W_1}(\mu_{1}(v)-\mu_{2}(v))\Big)\\
&=W_1(\mu_1,\mu_2)
\end{align*}
yielding the equation \eqref{eqn:OKP} as desired.
\end{proof}

\section{The idleness function is 3-piece linear}\label{sect: 3 linear}

In this section we will prove that, for any $x,y\in V$ such that $d(x,y)\ge 2$, the idleness function $p\mapsto \kappa_p(x,y)$ is piecewise linear with at most 3 linear parts. The proof follows the method from Theorem 5.2 in \cite{BCLMP17}, which prove in case $x\sim y$. The difference in the long-scale case lies in the following lemma. 

\begin{lemma} \label{lemma:d-2}
Let $G=(V,E)$ be a locally finite graph, and let $x,y\in V$ with $d(x,y)=\delta\ge 2$. Given $p\in (0,1]$, every optimal Kantorovich potential $\phi^*:V\rightarrow \mathbb{R}$ transporting $\mu_x^p$ to $\mu_y^p$ satisfies $$\delta-2 \le \phi^*(x)-\phi^*(y) \le \delta.$$ Moreover, if $p>\frac{1}{2}$, then $$\phi^*(x)-\phi^*(y)=\delta$$
\end{lemma}

\begin{proof}
Let $\pi^*\in\prod(\mu_x^p,\mu_y^p)$ and $\phi^*:V\rightarrow \mathbb{R}$ be an optimal plan and an optimal Kantorovich potential with respect to $W_1(\mu_x^p,\mu_y^p)$. For convenience, label the neighbours of $x$ by $x_1,...,x_{d_x}$ (where $d_x=\deg(x)$) and also label $x=x_0$.

Since $\sum\limits_{i=0}^{d_x} \pi^*(x_i,y)=\mu_y^p(y)=p>0$, we have $\pi^*(x_k,y)>0$ for some $0\le k \le d_x$. By complementary slackness theorem, 
$\phi^*(x_k)-\phi^*(y)=d(x_k,y)$. By Lipschitz and metric properties, we then have
$$\delta\ge \phi^*(x)-\phi^*(y)\ge (\phi^*(x_k)-1)-\phi^*(y) =d(x_k,y)-1\ge d(x,y)-2=\delta-2.$$
Now, assume $p>\frac{1}{2}$. Note that $$\sum\limits_{i=0}^{d_x} \pi^*(x_i,y)=\mu_y^p(y)=p>1-p=\sum\limits_{i=i}^{d_x} \mu_x^p(x_i)\ge \sum\limits_{i=i}^{d_x} \pi^*(x_i,y).$$ Hence, $\pi^*(x_0,y)>0$, on which the complementary slackness theorem implies that $$\phi^*(x)-\phi^*(y)=d(x,y)=\delta$$
\end{proof}

\begin{theorem} \label{thm:3linear}
Let $G=(V,E)$ be a locally finite graph, and let $x,y\in V$ with $d(x,y)=\delta\ge 2$. Then $p\mapsto \kappa_p(x,y)$ is concave and piecewise linear over $[0,1]$ with at most 3 linear parts.
\end{theorem}

\begin{proof}
For $\phi:V\rightarrow \mathbb{R}$, define $F(\phi):=\frac{1}{d_x}\sum\limits_{w\sim x} \phi(w)-\frac{1}{d_y}\sum\limits_{w\sim y} \phi(w)$, \\
and for $j\in\{\delta-2,\delta-1,\delta\}$, define a set $$A_j:=\left\{\phi:V\rightarrow \mathbb{Z} \bigg| \phi\in \textup{1-Lip},\ \phi(x)=j,\ \phi(y)=0\right\}.$$

Moreover, define a constant $c_j:=\sup\limits_{\phi\in A_j} F(\phi)$.
A linear function $f_j:\mathbb{R} \rightarrow \mathbb{R}$ is then defined by $f_j(t):=t\cdot j+(1-t)c_j$. It follows that

\begin{align} \label{3linearproof}
W_1(\mu_x^p,\mu_y^p)
&=\sup\limits_{\phi\in 1-Lip} \sum\limits_{w\in V} \phi(w)(\mu_x^p(w)-\mu_y^p(w)) \\
&\stackrel{\makebox[0pt]{Lemma \ref{lemma:d-2}} }{=}\sup\limits_{\substack{\phi\in 1-Lip\\\phi(x)-\phi(y)\in[\delta-2,\delta]}} \sum\limits_{w\in V} \phi(w)(\mu_x^p(w)-\mu_y^p(w)) \nonumber\\
&=\sup\limits_{\substack{\phi\in 1-Lip\\\phi:V\rightarrow \mathbb{Z}\\\phi(y)=0\\\phi(x)\in\{\delta-2,\delta-1,\delta\}}}\sum\limits_{w\in V} \phi(w)(\mu_x^p(w)-\mu_y^p(w)) \nonumber\\
&=\sup\limits_{\substack{\phi\in 1-Lip\\\phi:V\rightarrow \mathbb{Z}\\\phi(y)=0\\\phi(x)\in\{\delta-2,\delta-1,\delta\}}}
\bigg( p\phi(x)+\frac{1-p}{d_x}\sum\limits_{w\sim x} \phi(w)- \frac{1-p}{d_y}\sum\limits_{w\sim y} \phi(w) \bigg) \nonumber\\
&=\max\limits_{j\in\{\delta-2,\delta-1,\delta\}} \sup\limits_{\phi\in A_j} \left(p\cdot j+(1-p)F(\phi)\right) \nonumber\\
&=\max\limits_{j\in\{\delta-2,\delta-1,\delta\}} (p\cdot j+(1-p)c_j) \nonumber\\
&=\max\{f_{\delta-2}(p), f_{\delta-1}(p), f_{\delta}(p)\}, \nonumber
\end{align}
and therefore 
\begin{align} \label{eqn: idleness_3part}
\kappa_p(x,y)=1-\frac{1}{\delta}\max\{f_{\delta-2}(p), f_{\delta-1}(p), f_{\delta}(p)\}.
\end{align}

Hence, $p\mapsto \kappa_p(x,y)$ is concave and piecewise linear with at most 3 linear parts. 
\end{proof}

\begin{remark} \label{remark: last half linear}
For $p>\frac{1}{2}$, in the second line of equations \eqref{3linearproof}, the condition on the supremum can be replaced by $\phi(x)-\phi(y)=\delta$, due to the second half of Lemma \ref{lemma:d-2}. Doing so gives $W_1(\mu_x^p,\mu_y^p)=f_{\delta}(p)$ for all $p>\frac{1}{2}$. In other words, the idleness function $p\mapsto \kappa_p(x,y)$ has the last linear part (at least) on the interval $[\frac{1}{2}, 1]$. The same statement is also true in case $x\sim y$ (see \cite{BCLMP17}, Theorem 5.2). One immediate consequence is the simplification of Lin-Lu-Yau curvature:
\end{remark}

\begin{corollary} \label{cor:LLY}
Let $x\not=y \in V$. The Lin-Lu-Yau curvature satisfies $$\kappa_{LLY}(x,y)=\frac{\kappa_p(x,y)}{1-p}$$ for all $p\in [\frac{1}{2},1)$.
\end{corollary}

\section{Critical points of the idleness function}\label{sect: critical points}

In this section, we will discuss about the length of each linear part of the idleness function in terms of ``critical points''.

\begin{definition}[critical points]
Define \textit{critical points} (of $\kappa_{p}(x,y)$) to be the values $p^* \in (0,1)$ such that $$\lim\limits_{p\rightarrow p^*_-} \frac{\partial}{\partial p}\kappa_{p}(x,y) \not=\lim\limits_{p\rightarrow p^*_+} \frac{\partial}{\partial p}\kappa_{p}(x,y).$$
\end{definition}

In other words, critical points are the values of $p$ where the changes of slopes of the function $p\mapsto \kappa_{p}(x,y)$ happen. We may replace $\kappa_{p}(x,y)$ by $W_1(\mu_{x}^p,\mu_y^p)$ because they are closely related by the linear relation: $$\kappa_{p}(x,y)=1-\frac{W_1(\mu_{x}^p,\mu_y^p)}{d(x,y)}$$ and hence, they share the same critical points. Since the idleness function has at most 3 linear pieces, there are at most two critical points.
Our goal of this section is to determine the possible values of critical points. 

Next is the main theorem of this section, which gives an upper bound on the values of critical points. Such bound is sharp, as shown and explained in the later section.

\begin{theorem} \label{thm:onehalf}
Let $G=(V,E)$ be a locally finite graph, and let $x,y\in V$ with $d(x,y)=\delta\ge 2$. Let $p^*\in(0,1)$ be a critical point of $\kappa_{p}(x,y)$. Then 
$$p^* \le \frac{1}{2}-\frac{1}{2}\cdot\frac{d_x+d_y}{2d_xd_y-d_x-d_y}$$
\end{theorem}

\noindent The key of the proof lies in the following two lemmas. The first one compares the terms $c_j$'s introduced in the proof of Theorem \ref{thm:3linear}. The second one give an explicit formula for critical points in terms of $c_j$'s.

\begin{lemma} \label{lemma:cdelta}
With the same setup as above, 
\begin{align} \label{cdeltaeqn}
-1 < c_{\delta-2}-c_{\delta-1} \le c_{\delta-1}-c_{\delta} \le 1-\frac{1}{d_x}-\frac{1}{d_y}.
\end{align}
\end{lemma}

\noindent The proof of Lemma \ref{lemma:cdelta} is postponed towards the end of this section.

\begin{lemma} \label{lemma:intersectionpt}
With the same setup as above, define constants $p_1,p_2\in\mathbb{R}$ to be
\begin{align*}
p_1&:=\frac{c_{\delta-2}-c_{\delta-1}}{c_{\delta-2}-c_{\delta}+1}\\
p_2&:=\frac{c_{\delta-1}-c_{\delta}}{c_{\delta-1}-c_{\delta}+1}
\end{align*}
	
Then, for all $t\in\mathbb{R}$, 
\begin{align} \label{eqn:max_f}
\max\{f_{\delta-2}(t), f_{\delta-1}(t), f_{\delta}(t)\}=
\begin{cases}
f_{\delta-2}(t) &\textup{ if } t\le p_1\\
f_{\delta-1}(t) &\textup{ if } p_1\le t\le p_2\\
f_{\delta}(t) &\textup{ if } p_2\le t.
\end{cases}
\end{align}
	
Consequently, $p_1$ and $p_2$ are the only potentially critical points of $\kappa_p(x,y)$. Moreover, for each $i\in \{1,2\}$, $p_i$ is a critical point if and only if $p_i\in(0,1)$.
\end{lemma}

\begin{proof} [Proof of Lemma \ref{lemma:intersectionpt}]
First, note that the denominators of $p_1$ and $p_2$ are positive real numbers, due to Lemma \ref{lemma:cdelta}.
	
Next, we show that $p_1 \le p_2$. Consider the function $g:(-1,\infty)\rightarrow \mathbb{R}$ defined by $$g(t):=\frac{t}{t+1},$$ which is an increasing function on $t$. 

Note that $p_1=g(c_{\delta-2}-c_{\delta-1})$ and $p_2=g(c_{\delta-1}-c_{\delta})$. Hence, Lemma \ref{lemma:cdelta} implies $$p_1= g(c_{\delta-2}-c_{\delta-1})\le g(c_{\delta-1}-c_{\delta})=p_2$$
	
Next, we compare between $f_{\delta-2}$ and $f_{\delta-1}$. From the definition $f_j(t)=t\cdot j+(1-t)c_j$, we have
\begin{align*}
f_{\delta-1}(t) \ge f_{\delta-2}(t)  
&\Leftrightarrow t(c_{\delta-2}-c_{\delta-1}+1) \ge c_{\delta-2}-c_{\delta-1} \\
&\Leftrightarrow t\ge \frac{c_{\delta-2}-c_{\delta-1}}{c_{\delta-2}-c_{\delta-1}+1}=p_1
\end{align*}
Similarly, comparison between $f_{\delta-1}$ and $f_{\delta}$ gives:
\begin{align*}
f_{\delta}(t) \ge f_{\delta-1}(t)  &\Leftrightarrow t\ge \frac{c_{\delta-1}-c_{\delta}}{c_{\delta-1}-c_{\delta}+1}=p_2.
\end{align*}
By the above comparisons, we can then conclude the equation

$$\max\{f_{\delta-2}(t), f_{\delta-1}(t), f_{\delta}(t)\}=
\begin{cases}
f_{\delta-2}(t) &\textup{ if } t\le p_1\\
f_{\delta-1}(t) &\textup{ if } p_1\le t\le p_2\\
f_{\delta}(t) &\textup{ if } p_2\le t.
\end{cases}$$
as desired. Moreover, one can check that the slopes of $f_{\delta-2}$,$f_{\delta-1}$, and $f_{\delta}$ are all different:$$\frac{\partial}{\partial t}f_{\delta-2} < \frac{\partial}{\partial t}f_{\delta-1} <\frac{\partial}{\partial t}f_{\delta}.$$
	
The second statement in the lemma then immediately follows by renaming the variable $t$ as $p$ with a further restriction $p\in [0,1]$, and recalling the equation \eqref{eqn: idleness_3part}:
$$\kappa_p(x,y)=1-\frac{1}{\delta}\max\{f_{\delta-2}(p), f_{\delta-1}(p), f_{\delta}(p)\}.$$
\end{proof}

\begin{proof} [Proof of Theorem \ref{thm:onehalf}]
Recall the function $g$ defined in the proof of Lemma \ref{lemma:intersectionpt}. The monotonicity of $g$ together with Lemma \ref{lemma:cdelta} implies that
$$p_1\le p_2= g(c_{\delta-1}-c_{\delta}) \le g(1-\frac{1}{d_x}-\frac{1}{d_y})=\frac{1}{2}-\frac{1}{2}\cdot\frac{d_x+d_y}{2d_xd_y-d_x-d_y},$$
which concludes the theorem.
\end{proof}

Now we come back to prove Lemma \ref{lemma:cdelta}.

\begin{proof}[Proof of Lemma \ref{lemma:cdelta}]
First, we prove the rightmost inequality: $$c_{\delta-1}-c_{\delta} \le 1-\frac{1}{d_x}-\frac{1}{d_y}.$$
	
Consider $\phi^*:V\rightarrow \mathbb{Z}$ such that $\phi^*\in A_{\delta-1}$ and $F(\phi^*)=\sup\limits_{\phi\in A_{\delta-1}}F(\phi)=:c_{\delta-1}$. We will give a partial ordering to the set of vertices $V$ by the following rule: for $a,b\in V$ such that $a\sim b$, if $\phi^*(a)-\phi^*(b)=1$, then call such $b$ a \textit{child} of $a$. Moreover, for $a,b\in V$, give the ordering $b \prec a$ iff $b$ is a \textit{descendant} of $a$, that is there exist $a=b_0,b_1,b_2,...,b_n=b \in V$ such that $b_{i}$ is a child of $b_{i+1}$ for all $0\le i<n$. Lipschitz property of $\phi^*$ implies that 
\begin{align} \label{eqn:po_relation}
\phi^*(a)-\phi^*(b)=d(a,b) \textup{ for all } b \preceq a.
\end{align}
In particular, since $\phi^*(x)-\phi^*(y)=\delta-1 <d(x,y)$, it implies that $y\not\preceq x$.
	
Define a set of vertices $V_x\subset V$ by 
\begin{align*}
V_x:=\{w\in V\Big| w\preceq x \}.
\end{align*}
	Now, define a function $\phi':V\rightarrow \mathbb{Z}$ by
\begin{align*}
\phi'(w):=\begin{cases}
\phi^*(w)+1 &\textup{ if } w\in V_x, \\ \phi^*(w) &\textup{ otherwise}.
\end{cases}
\end{align*}
	
We will now show that $\phi'$ is $1$-Lipschitz. It is sufficient to show that $\phi'(w)-\phi'(z) \le 1$ for any neighbours $w\stackrel{G}{\sim} z\in V$. By definition of $\phi'$, we have
\begin{align*}
\phi'(w)-\phi'(z) &= (\phi^*(w)+\mathbbm{1}_{V_x}(w))-(\phi^*(z)+\mathbbm{1}_{V_x}(z)) \\
&=\phi^*(w)-\phi^*(z) + \mathbbm{1}_{V_x}(w)-\mathbbm{1}_{V_x}(z) \\
&\le 1+\mathbbm{1}_{V_x}(w)-\mathbbm{1}_{V_x}(z)
\end{align*}
which is less than or equal to $1$, except when $\phi^*(w)-\phi^*(z)=1$ and $\mathbbm{1}_{V_x}(w)=1$ and $\mathbbm{1}_{V_x}(z)=0$, simultaneously. These exception conditions would imply that $z$ is a child of $w$, and $w\preceq x$, and $z\not\preceq x$, which is impossible as it contradicts to the transitivity of partial ordering. Therefore, $\phi'$ is $1$-Lipschitz as desired. Moreover, $\phi'\in A_{\delta}$ (because $\phi'(x)=\phi^*(x)+1=\delta$ and $\phi'(y)=\phi^*(y)=0$ since $y\not\in V_x$). 
	
Comparison between $\phi^*$ and $\phi'$ gives
\begin{align}\label{cdelta1}
c_{\delta} = \sup\limits_{\phi\in A_{\delta}} F(\phi) \ge F(\phi')
&= \frac{1}{d_x}\sum\limits_{w\sim x}\phi'(w)-\frac{1}{d_y}\sum\limits_{w\sim y}\phi'(w) \\
&= \frac{1}{d_x}\sum\limits_{w\sim x}\phi^*(w)-\frac{1}{d_y}\sum\limits_{w\sim y}\phi^*(w) + \frac{1}{d_x}\sum\limits_{\substack{w\sim x\\w\in V_x}} 1 - \frac{1}{d_y}\sum\limits_{\substack{w\sim y\\w\in V_x}} 1 \nonumber \\
&= F(\phi^*)+\frac{1}{d_x}|N(x)\cap V_x|-\frac{1}{d_y}|N(y) \cap V_x| \nonumber \\
&= c_{\delta-1}+\frac{1}{d_x}|N(x)\cap V_x|-\frac{1}{d_y}|N(y) \cap V_x| \nonumber
\end{align}
where $N(x)$ and $N(y)$ are the sets of neighbours of $x$ and of $y$, respectively.
	
A simple bound on \eqref{cdelta1} will give
$$c_{\delta}-c_{\delta-1} \ge \frac{1}{d_x}|N(x)\cap V_x|-\frac{1}{d_y}|N(y) \cap V_x| \ge \frac{1}{d_x}(0)-\frac{1}{d_y}(d_y)=-1.$$
	
\noindent However, this inequality can be improved by the following 3-case separation.
\begin{itemize}
\item[$\bullet$]Case 1: $N(x)\cap V_x \not= \emptyset$ and $N(y)\cap V_x \not= N(y)$.\\
Then
\begin{align*}
c_{\delta}-c_{\delta-1} 
&\ge \frac{1}{d_x}|N(x)\cap V_x|-\frac{1}{d_y}|N(y) \cap V_x| \\ 
&\ge \frac{1}{d_x}(1)-\frac{1}{d_y}(d_y-1)=-1+\frac{1}{d_x}+\frac{1}{d_y}.
\end{align*}
		
\item[$\bullet$]Case 2: $N(x)\cap V_x = \emptyset$.\\
It means that $x$ has no child and hence no descendant, i.e. $V_x=\{x\}$. Thus
$$c_{\delta}-c_{\delta-1} \ge \frac{1}{d_x}|N(x)\cap V_x|-\frac{1}{d_y}|N(y) \cap V_x| = 0.$$

\item[$\bullet$]Case 3: $N(y)\cap V_x = N(y)$\\
It means that $y'\prec x$ for all neighbours $y'$ of $y$. We now define a new function $\phi'':V\rightarrow \mathbb{Z}$ by
\begin{align*}
\phi''(w):=\begin{cases}
\phi^*(w)+1 &\textup{ if } w\not=y, \\ \phi^*(w) &\textup{ if } w=y.
\end{cases}
\end{align*} which is 1-Lipschitz and in $A_{\delta}$ (similar as to how $\phi'$ 1-Lipschitz and in $A_{\delta}$). It follows that	
\begin{align*}
c_{\delta} = \sup\limits_{\phi\in A_{\delta}} F(\phi) \ge F(\phi'')
&= \frac{1}{d_x}\sum\limits_{w\sim x}\phi''(w)-\frac{1}{d_y}\sum\limits_{w\sim y}\phi''(w) \\
&= \frac{1}{d_x}\sum\limits_{w\sim x}(\phi^*(w)+1)-\frac{1}{d_y}\sum\limits_{w\sim y}(\phi^*(w)+1) \nonumber \\
&= F(\phi^*)= c_{\delta-1} \nonumber
\end{align*}
Hence, $c_{\delta}-c_{\delta-1} \ge 0.$
\end{itemize}
	
From the three cases above, we can conclude the rightmost inequality in \eqref{cdeltaeqn}: $$c_{\delta-1}-c_{\delta} \le 1-\frac{1}{d_x}-\frac{1}{d_y}.$$
	
Next, we prove the leftmost inequality:
$$-1\le c_{\delta-2}-c_{\delta-1}$$ with a similar method as above.
	
Define another set of vertices $\widetilde{V}_x \subset V$ by
\begin{align*}
\widetilde{V}_x:=\{w\in V\Big| x \preceq w \},
\end{align*}
	
\noindent Then define a function $\widetilde{\phi}:V\rightarrow \mathbb{Z}$ by
\begin{align*}
\widetilde{\phi}(w):=\begin{cases}
\phi^*(w)-1 &\textup{ if } w\in \widetilde{V}_x, \\ \phi^*(w) &\textup{ otherwise}.
\end{cases}
\end{align*}
	
By similar arguments, the function $\widetilde{\phi}$ is also $1$-Lipschitz, and it is in $A_{\delta-2}$ (because $\widetilde{\phi}(x)=\phi^*(x)-1=\delta-2$ and $\phi'(y)=\phi^*(y)=0$ since $y\not\in \widetilde{V}_x$). Comparison between $\phi^*$ and $\widetilde{\phi}$ gives
\begin{align}\label{cdelta2}
c_{\delta-2} = \sup\limits_{\phi\in A_{\delta-2}} F(\phi) \ge F(\widetilde{\phi})
&= \frac{1}{d_x}\sum\limits_{w\sim x}\widetilde{\phi}(w)-\frac{1}{d_y}\sum\limits_{w\sim y}\widetilde{\phi}(w) \\
&= \frac{1}{d_x}\sum\limits_{w\sim x}\phi^*(w)-\frac{1}{d_y}\sum\limits_{w\sim y}\phi^*(w) - \frac{1}{d_x}\sum\limits_{\substack{w\sim x\\w\in \widetilde{V}_x}} 1 + \frac{1}{d_y}\sum\limits_{\substack{w\sim y\\w\in \widetilde{V}_x}} 1 \nonumber \\
&= F(\phi^*)-\frac{1}{d_x}|N(x)\cap \widetilde{V}_x|+\frac{1}{d_y}|N(y) \cap \widetilde{V}_x| \nonumber \\
&= c_{\delta-1}-\frac{1}{d_x}|N(x)\cap \widetilde{V}_x|+\frac{1}{d_y}|N(y) \cap \widetilde{V}_x| \nonumber\\
&\ge c_{\delta-1}-\frac{1}{d_x}(d_x)+\frac{1}{d_y}(0) \nonumber\\
&\ge c_{\delta-1}-1 \nonumber
\end{align}

By considering a geodesic from $x$ to $y$, namely $x=v_0\sim v_1 \sim v_2 \sim...\sim v_{\delta}=y$, we have that $v_1\in N(x)$ and $$\phi^*(v_1)=\phi^*(v_1)-\phi^*(y)\le d(v_1,y)=\delta-1=\phi^*(x),$$ which implies that $x\not\preceq v_1$, i.e. $v_1\not\in \widetilde{V}_x$.
Hence, $|N(x)\cap \widetilde{V}_x| < d_x$, and \eqref{cdelta2} then gives $$c_{\delta-2}>c_{\delta-1}-\frac{1}{d_x}(d_x)+\frac{1}{d_y}(0)=c_{\delta-1}-1$$
yielding the leftmost inequality in \eqref{cdeltaeqn}.
	
Lastly, we will prove the middle inequality in \eqref{cdeltaeqn}, or equivalently,
$$c_{\delta-2}+c_{\delta} \le 2c_{\delta-1}.$$
	
Let $\phi_{\delta-2}\in A_{\delta-2}$ and $\phi_{\delta}\in A_{\delta}$ be two 1-Lipschitz functions such that $c_{\delta-2}=F(\phi_{\delta-2})$ and $c_{\delta}=F(\phi_{\delta})$. Consider the function $\Phi:=\frac{1}{2}(\phi_{\delta-2}+\phi_{\delta}).$ From the definition, we know that $\Phi$ is also 1-Lipschitz, and $\Phi(v)\in \mathbb{Z}/2$ for all $v\in V$, and $\Phi(x)=\delta-1$ and $\Phi(y)=0$.
Therefore $$\frac{1}{2}(c_{\delta-2}+c_{\delta}) = \frac{F(\phi_{\delta-2})+F(\phi_{\delta})}{2}=F\Big(\frac{\phi_{\delta-2}+\phi_{\delta}}{2}\Big)=F(\Phi) \le \sup\limits_{\phi\in A_j[\mathbb{Z}/2]}F(\phi)$$

where $A_j[\mathbb{Z}/2]:=\left\{\phi:V\rightarrow \mathbb{Z}/2 \bigg| \phi\in \textup{1-Lip},\ \phi(x)=j,\ \phi(y)=0\right\}$ is defined in parallel to the previously defined $A_j=\left\{\phi:V\rightarrow \mathbb{Z} \bigg| \phi\in \textup{1-Lip},\ \phi(x)=j,\ \phi(y)=0\right\}$.
	
Lastly, we are left to show that 
$$\sup\limits_{\phi\in A_{\delta-1}[\mathbb{Z}/2]}F(\phi)
=\sup\limits_{\phi\in A_{\delta-1}}F(\phi)=:c_{\delta-1},$$
where the proof is very similar to the proof of Lemma 5.1 in \cite{BCLMP17}.

The inequality $\sup\limits_{\phi\in A_{\delta-1}[\mathbb{Z}/2]}F(\phi) \ge c_{\delta-1}$ is trivial since $A_{\delta-1} \subseteq A_{\delta-1}[\mathbb{Z}/2]$.

On the other hand, choose a function $\phi_{\delta-1}\in A_{\delta-1}[\mathbb{Z}/2]$ such that $F(\phi_{\delta-1}) =\sup\limits_{\phi\in A_{\delta-1}[\mathbb{Z}/2]}F(\phi).$ Note that $\phi_{\delta-1}(v)=\frac{1}{2}(\lfloor \phi_{\delta-1}(v) \rfloor + \lceil \phi_{\delta-1}(v) \rceil)$ for all $v\in V$, and $\lfloor \phi_{\delta-1} \rfloor ,\lceil \phi_{\delta-1} \rceil \in A_{\delta-1}$. Therefore, $$\sup\limits_{\phi\in A_{\delta-1}[\mathbb{Z}/2]}F(\phi)=F(\phi_{\delta-1})=F\Big(\frac{\lfloor \phi_{\delta-1} \rfloor + \lceil \phi_{\delta-1} \rceil}{2}\Big)=\frac{F(\lfloor \phi_{\delta-1} \rfloor)+F(\lceil \phi_{\delta-1} \rceil)}{2} \le c_{\delta-1}$$ as desired. 
\end{proof}

\begin{remark}[Length of the first linear part] \label{remark:firstcrit}
Note that each $c_{j}\in \mathbb{Z}/l$ where $l:=\textup{lcm}(d_x,d_y)$. Therefore, $p_1$ and $p_2$ must be in the form of $\frac{\frac{a}{l}}{\frac{a}{l}+1}=\frac{a}{a+l}$ for some $a\in \mathbb{Z}$. Hence, the least possible value for a positive critical point is $\frac{1}{1+l}$. In other words, first linear part of the function $p\mapsto \kappa_p(x,y)$ is at least the interval $[0,\frac{1}{1+\textup{lcm}(d_x,d_y)}]$.
\end{remark}

\begin{remark}[Length of the last linear part] \label{remark:critpts}
Theorem \ref{thm:onehalf} says that the last linear part of the function $p\mapsto \kappa_p(x,y)$ is at least the interval $$\Big[ \frac{1}{2}-\frac{1}{2}\cdot\frac{d_x+d_y}{2d_xd_y-d_x-d_y},1\Big].$$
In a special case that vertices $x$ and $y$ have the same degree $d_x=d_y=D$, each critical point $p^*$ of $\kappa_p(x,y)$ satisfies 
\begin{align} \label{eqn:critpt_sharp}
p^*\le \frac{1}{2}-\frac{1}{2}\cdot\frac{2D}{2D^2-2D}=\frac{D-2}{2D-2}.
\end{align}
Moreover, from the definition, $c_j\in \mathbb{Z}/{D}$, so $p_1,p_2$ must be in the form of $\frac{a}{D+a}$ for some integer $1\le a\le D-2$.
\end{remark}

Next section provides an example of a graph with $d_x=d_y=D$ where the inequality \eqref{eqn:critpt_sharp} holds sharp, i.e. a critical point occurs exactly at $\frac{D-2}{2D-2}$.

\newpage
\section{An important family of examples} \label{sect: example}

In this section we aim to construct a graph $G=(V,E)$ with points $x,y\in G$ such that $d(x,y)\ge 2$ and the idleness function $p\mapsto \kappa_p(x,y)$ has three linear pieces and has one critical point as large as the one mentioned in \eqref{eqn:critpt_sharp}. 

Let $m,n,k$ be arbitrary natural numbers (including zero). Define vertices of $G$ to be $$V:=\{x,y\}\cup \{x_{0},x_{1},y_{0},y_{1}\} \cup \bigcup\limits_{i=1}^{m} \{x'_i,y'_i,v_i,w_i\} \cup \bigcup\limits_{i=1}^{n} \{x''_i,y''_i,z_i\} \cup \bigcup\limits_{i=1}^{k} \{x'''_i,y'''_i\},$$
and define edges of $G$ to be
\begin{align*}
E:=
&\{(x,x_0),(x_0,y_0),(y_0,y)\} \ \cup \ \{(x,x_1),(x_1,y_1),(y_1,y)\}  \ \cup \ \\
& \bigcup\limits_{i=1}^{m} \{(x,x'_i),(x'_i,v_i),(v_i,w_i),(w_i,y'_i),(y'_i,y)\} \ \cup \\
&\bigcup\limits_{i=1}^{n} \{(x,x''_i),(x''_i,z_i),(z_i,y''_i),(y''_i,y)\} \ \cup \\
&\bigcup\limits_{i=1}^{k} \{(x,x'''_i),(x'''_i,y'''_i),(y'''_i,y)\} \ \cup \\
& \bigcup\limits_{i=1}^{m} \{(x_0,y'_i), (x'_i,y_1)\} \ \cup \ 
\bigcup\limits_{i=1}^{n} \{(x_0,y''_i), (x''_i,y_1)\} \ \cup \ 
\bigcup\limits_{i=1}^{k} \{(x_0,y'''_i), (x'''_i,y_1)\}
\end{align*}
If $m$, $n$ or $k$ is zero, we simply remove the related vertices and edges.

\begin{figure}[h!] 
	\begin{tikzpicture} [scale=0.7]
	
	\foreach \x in {0,3,5,7,10} {
		\drawLinewithBG{5,0}{\x,3}{black};
		\drawLinewithBG{5,12}{\x,9}{black};
	}
	
	\foreach \x in {0,3,5,7,10} {
		\drawLinewithBG{\x,3}{\x,9}{blue};
	}
	
	\foreach \x in {3,5,7} {
		\drawLinewithBG{\x,3}{0,9}{red};
		\drawLinewithBG{\x,9}{10,3}{red};	
	}
	
	\drawlabelnode{5,12}{above right}{$x$};
	\drawlabelnode{5,0}{below right}{$y$}
	\drawlabelnode{0,9}{left}{$x_0$};
	\drawlabelnode{10,9}{right}{$x_1$};
	\drawlabelnode{0,3}{left}{$y_0$};
	\drawlabelnode{10,3}{right}{$y_1$};
	
	\drawlabelnode{5,9}{right}{$x''_n$};
	\drawlabelnode{5,6}{right}{$z_n$};
	\drawlabelnode{5,3}{right}{$y''_n$};
	
	\drawlabelnode{3,9}{right}{$x'_m$};
	\drawlabelnode{3,7}{right}{$u_m$};
	\drawlabelnode{3,5}{right}{$v_m$};
	\drawlabelnode{3,3}{right}{$y'_m$};
	
	\drawlabelnode{7,9}{right}{$x'''_k$};
	\drawlabelnode{7,3}{right}{$y'''_k$};
	\end{tikzpicture}
	\caption{The constructed graph $G$ with $m=n=k=1$.}
	\label{fig:graph_mnk}
\end{figure}
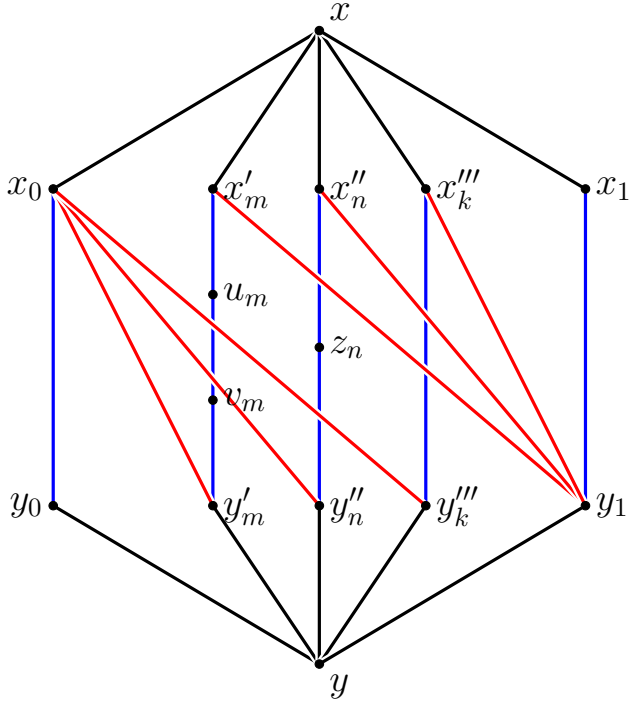

The graph $G$ is shown in Figure \ref{fig:graph_mnk} in case $m=n=k=1$ (but the indexes $m,n,k$ are kept in the labelling for clarity).

In the constructed graph $G$, $d(x,y)=3$ and $d_x=d_y=D=2+m+n+k$. Our goal is to show that the function $p\mapsto W_1(\mu_x^p,\mu_y^p)$ has its critical points at $\frac{m}{D+m}$ and $\frac{m+n}{D+m+n}$. In particular, if $k=0$, then the larger critical point coincides with $\frac{D-2}{2D-2}$, the maximum possible value mentioned in Remark \ref{remark:critpts}.)\\

Recall the definitions, for $j\in\{1,2,3\}$,
$$c_j=\sup\limits_{\phi\in A_j} F(\phi)=\sup\limits_{\phi\in A_j} \frac{1}{D}\bigg( \sum\limits_{w\sim x}\phi(w)-\sum\limits_{w\sim y}\phi(w)\bigg)$$ and
$$f_j(p)=p\cdot j+(1-p)c_j,$$ and 
$$W_1(\mu_x^p,\mu_y^p)=\max\{f_1(p),f_2(p),f_3(p)\}.$$

We need to calculate the value of $c_j$ for each $j\in\{1,2,3\}$. First, we start by giving a lower bound to $c_j$ by choosing an appropriate function $\phi_j\in A_j$ for each $j$.

\begin{table}[ht]
\centering 
\begin{tabular}{|c | c c | c c c c| c c c c| c c c| c c |}
\hline
		& \multicolumn{15}{|c|}{z} \\
\cline{2-16}
\raisebox{-1.0ex}{$\phi_j(z)$} &\raisebox{-1.0ex}{$x$} &\raisebox{-1.0ex}{$y$} &\raisebox{-1.0ex}{$x_0$} &\raisebox{-1.0ex}{$x_1$} &\raisebox{-1.0ex}{$y_0$} &\raisebox{-1.0ex}{$y_1$} &$x'_i$ &$u_i$ &$v_i$ &$y'_i$ &$x''_i$ &$z_i$ &$y''_i$ &$x'''_i$ &$y'''_i$  \\
& & & & & & & \multicolumn{4}{c}{$1\le i \le m$} & \multicolumn{3}{|c|}{$1\le i \le n$} & \multicolumn{2}{|c|}{$1\le i \le k$}\\[0.5ex]
\hline
$\phi_1$ &1 &0& 0&2 &-1 &1 &2 &1 &0 &-1 &2 &1 &0 &1 &0 \\[1.0ex]
$\phi_2$ &2 &0& 1&2 &0 &1 &2 &1 &0 &0 &2 &1 &0 &1 &0 \\[1.0ex]
$\phi_3$ &3 &0& 2&2 &1 &1 &2 &1 &1 &1 &2 &1 &1 &2 &1 \\[1.0ex]
\hline
\end{tabular}
\caption{Functions $\phi_j$ evaluated at the vertices of $G$}
\label{table:potentials_123}
\end{table}
Define functions $\phi_1,\phi_2,\phi_3:V\rightarrow \mathbb{Z}$ as in Table \ref{table:potentials_123}. It can be easily checked that $\phi_j\in A_j$ for $j\in\{1,2,3\}$, and hence we obtains three following inequalities:
\begin{flalign} 
c_1 &\ge F(\phi_1)=\frac{1}{D}\bigg((2m+2n+k+2)-(-m)\bigg)=\frac{3m+2n+k+2}{D} \label{c1lowerbound} \\
c_2 &\ge F(\phi_2)=\frac{1}{D}\bigg((2m+2n+k+3)-(1)\bigg)=\frac{2m+2n+k+2}{D} \label{c2lowerbound}\\
c_3 &\ge F(\phi_3)=\frac{1}{D}\bigg((2m+2n+2k+4)-(m+n+k+2)\bigg)=\frac{m+n+k+2}{D}=1 \label{c3lowerbound}
\end{flalign}

Next, we give an upper bound to $c_1,\ c_2,\ c_3$ by calculating the costs of transport plans $\pi_1, \pi_2, \pi_3$ from $\mu_x^p$ to $\mu_y^p$ (with idleness $p=0,\ \frac{m}{D+m},\ \frac{m+n}{D+m+n}$, respectively). The plans $\pi_1, \pi_2, \pi_3$ are constructed as in Table \ref{table:plans_123}. \newpage

\begin{table}[ht]
	\centering 
	\begin{tabular}{|c| c| c| c c c c c c|}
		\hline
		\multicolumn{3}{|c|}{} & \multicolumn{6}{|c|}{z} \\
		\cline{4-9}
		\multicolumn{3}{|c|}{$\pi_j(w,z)$} &\raisebox{-1.0ex}{$\ y\ $} &\raisebox{-1.0ex}{$\ y_0\ $} &\raisebox{-1.0ex}{$\ y_1\ $} &$y'_i$ &$y''_i$ &$y'''_i$ \\
		\multicolumn{3}{|c|}{} & & & & {$1\le i \le m$} & {$1\le i \le n$} & {$1\le i \le k$}\\[0.5ex]
		\hline 
		& & $x$ & & & & & & \\
		& & $x_0$ & & $\frac{1}{D}$ & & & & \\
		\raisebox{-1.0ex}{$\pi_1$} & \raisebox{-1.0ex}{$w$} & $x_1$ & & & $\frac{1}{D}$ & & & \\
		& & $x_i'\ (1\le i \le m)$ & & & & $\frac{1}{D}$ & & \\
		& & $x_i''\ (1\le i \le n)$ & & & & & $\frac{1}{D}$ & \\
		& & $x_i'''\ (1\le i \le k)$ & & & & & & $\frac{1}{D}$ \\[1.0ex]
		\hline
		
		& & $x$ & & & & $\frac{1}{D+m}$ & & \\
		& & $x_0$ & & $\frac{1}{D+m}$ & & & & \\
		\raisebox{-1.0ex}{$\pi_2$} & \raisebox{-1.0ex}{$w$} & $x_1$ & & & $\frac{1}{D+m}$ & & & \\
		& & $x_i'\ (1\le i \le m)$ & $\frac{1}{D+m}$ & & & & & \\
		& & $x_i''\ (1\le i \le n)$ & & & & & $\frac{1}{D+m}$ & \\
		& & $x_i'''\ (1\le i \le k)$ & & & & & & $\frac{1}{D+m}$ \\[1.0ex]
		\hline	
		
		& & $x$ & & & & $\frac{1}{D+m+n}$ & $\frac{1}{D+m+n}$ & \\
		& & $x_0$ & & $\frac{1}{D+m+n}$ & & & & \\
		\raisebox{-1.0ex}{$\pi_3$} & \raisebox{-1.0ex}{$w$} & $x_1$ & & & $\frac{1}{D+m+n}$ & & & \\
		& & $x_i'\ (1\le i \le m)$ & $\frac{1}{D+m+n}$ & & & & & \\
		& & $x_i''\ (1\le i \le n)$ & $\frac{1}{D+m+n}$ & & & & & \\
		& & $x_i'''\ (1\le i \le k)$ & & & & & & $\frac{1}{D+m+n}$ \\[1.0ex]
		\hline		
		
	\end{tabular}
	\caption{Transport plans $\pi_j$ evaluated (non-vanishingly) at pairs of vertices of $G$}
	\label{table:plans_123}
\end{table}

It is straightforward to check that $\pi_1\in\prod(\mu_x^0,\mu_y^0)$, $\pi_2\in\prod(\mu_x^{\frac{m}{D+m}},\mu_y^{\frac{m}{D+m}})$, and $\pi_3\in\prod(\mu_x^{\frac{m+n}{D+m+n}},\mu_y^{\frac{m+n}{D+m+n}})$, and therefore


\begin{align} \label{c1upperbound}
c_1=f_1(0)\le W_1(\mu_x^0,\mu_y^0) &\le \sum\limits_{(w,z)\in V^2} \pi_1(w,z)d(w,z) \nonumber \\
&=\frac{1}{D}\bigg(2+3m+2n+k\bigg). 
\end{align}

%

\begin{align*}
\frac{2m}{D+m}+\frac{D}{D+m}c_2=f_2(\frac{m}{D+m})
&\le W_1(\mu_x^\frac{m}{D+m},\mu_y^\frac{m}{D+m}) \\
&\le \sum\limits_{(w,z)\in V^2} \pi_2(w,z) d(w,z)\\
&=\frac{1}{D+m}\bigg(2+4m+2n+k\bigg).
\end{align*}
which implies
\begin{align} \label{c2upperbound}
c_2\le \frac{2m+2n+k+2}{D}.
\end{align}

%

\begin{align*}
\frac{3(m+n)}{D+m+n}+\frac{D}{D+m+n}c_3=f_3(\frac{m+n}{D+m+n})
&\le W_1(\mu_x^\frac{m+n}{D+m+n},\mu_y^\frac{m+n}{D+m+n}) \\
&\le \sum\limits_{(w,z)\in V^2} \pi_3(w,z) d(w,z)\\
&=\frac{1}{D+m+n}\bigg(2+4m+4n+k\bigg).
\end{align*}
which implies
\begin{align} \label{c3upperbound}
c_3\le \frac{m+n+k+2}{D}=1.
\end{align}

By comparing \eqref{c1lowerbound},\eqref{c2lowerbound},\eqref{c3lowerbound} to \eqref{c1upperbound},\eqref{c2upperbound},\eqref{c3upperbound}, we know that the exact values of $c_j$'s:

$$c_1= \frac{3m+2n+k+2}{D} \ ; \ c_2= \frac{2m+2n+k+2}{D} \ ; \ c_3= 1.$$

Lemma \ref{lemma:intersectionpt} then give a formula for $W_1(\mu_x^p,\mu_y^p)$ that
\begin{align*}
W_1(\mu_x^p,\mu_y^p)=
\begin{cases}
f_{1}(p) &\textup{ if } 0\le p\le p_1\\
f_{2}(p) &\textup{ if } p_1\le p\le p_2\\
f_{3}(p) &\textup{ if } p_2\le p\le 1.
\end{cases}
\end{align*}
where the critical points are
$$p_1=\frac{c_1-c_2}{c_1-c_2+1}=\frac{m}{D+m} \textup{ and } p_2=\frac{c_2-c_3}{c_2-c_3+1}=\frac{m+n}{D+m+n},$$
as we would like to verify.

\newpage

\section{The Cartesian product}\label{sect: cartesian}
In \cite{LLY11} the authors proved the following results on the curvature of Cartesian products of graphs:
\begin{theorem}[\cite{LLY11}]
Let  $G=(V_{G},E_{G})$ be a $d_{G}$-regular graph and $H=(V_{H},E_{H})$ be a $d_{H}$-regular graph. Let $x_{1},x_{2}\in V_{G}$ with $x_{1}\sim x_{2}$ and $y\in V_{H}$. Then
\begin{align*}
\kappa^{G\times H}_{LLY}((x_{1},y),(x_{2},y)) & = \frac{d_{G}}{d_{G}+d_{H}} \kappa^{G}(x_{1},x_{2}),
\\
\kappa^{G\times H}_{0}((x_{1},y),(x_{2},y)) & = \frac{d_{G}}{d_{G}+d_{H}} \kappa^{G}_{0}(x_{1},x_{2}).
\end{align*}
\end{theorem}

We now extend this result to long scale curvature.

\begin{theorem} \label{thm:LLY_cart}
Let  $G=(V_{G},E_{G})$ be a $D_{G}$-regular graph and $H=(V_{H},E_{H})$ be a $D_{H}$-regular graph. Let $x_{1},x_{2}\in V_{G}$ and $y_{1}, y_{2}\in V_{H}$. Then
\begin{align*}
\kappa^{G\times H}_{LLY}((x_{1},y_{1}),(x_{2},y_{2})) & = \frac{D_{G}d(x_{1},x_{2})\kappa^{G}_{LLY}(x_{1},x_{2})+D_{H}d(y_{1},y_{2})\kappa^{H}_{LLY}(y_{1},y_{2})}{(D_{G}+D_{H})(d(x_{1},x_{2})+d(y_{1},y_{2}))}.
\end{align*}

Furthermore, for all $p\in [\frac{1}{2},1)$, we have 
\begin{align*}
\kappa^{G\times H}_{p}((x_{1},y_{1}),(x_{2},y_{2})) & = \frac{D_{G}d(x_{1},x_{2})\kappa^{G}_{p}(x_{1},x_{2})+D_{H}d(y_{1},y_{2})\kappa^{H}_{p}(y_{1},y_{2})}{(D_{G}+D_{H})(d(x_{1},x_{2})+d(y_{1},y_{2}))}.
\end{align*}
\end{theorem}

Here we use the notation $D_G,D_H$, instead of $d_G,d_H$ for the degree to distinguish them from the distance function $d(\cdot,\cdot)$. Moreover, we use convention $d(x_1,x_2)K_{LLY}^G(x_1,x_2)=0$ in case $x_1=x_2$, and $d(y_1,y_2)K_{LLY}^H(y_1,y_2)=0$ in case $y_1=y_2$. Before proving the theorem, we introduce a lemma stating that the sum of 1-Lipschitz functions on two different graphs is a 1-Lipschitz function on the cartesian-product graph. 

\begin{lemma} \label{lemma:cart_lip}
Let $G=(V_{G},E_{G})$ and $H=(V_{H},E_{H})$ be two locally finite graphs. Suppose $\phi_G:V_G\rightarrow \mathbb{R}$ and $\phi_H:V_H\rightarrow \mathbb{R}$ are 1-Lipschitz functions on $G$ and $H$, respectively. Then the function $\phi:V_{G\times H} \rightarrow \mathbb{R}$ defined by 
$$\phi((w,z)):=\phi_G(w)+\phi_H(z) \textup{ for all } w\in V_G, z\in V_H$$
is also 1-Lipschitz function on the cartesian product $G\times H$.
\end{lemma}
\begin{proof}[Proof of Lemma \ref{lemma:cart_lip}]
Let $w_1,w_2\in V_G$ and $z_1,z_2\in V_H$. By applying 1-Lipschitz properties of $\phi_G$ and $\phi_H$, we obtain
\begin{align*}
\phi((w_1,z_1))-\phi((w_2,z_2)) &= \phi_G(w_1)+\phi_H(z_1) - \phi_G(w_2)-\phi_H(z_2) \\ 
&\le d(w_1,w_2) +d(z_1,z_2) \\ &= d((w_1,z_1),(w_2,z_2))
\end{align*}
yielding the lemma.
\end{proof}

\begin{proof}[Proof of Theorem \ref{thm:LLY_cart}] 
For idleness $p\in[0,1]$, define idleness $\lambda,\lambda'\in[0,1]$ to be 
\begin{align} \label{eqn:choose_idleness}
\lambda=\frac{pD_G+D_H}{D_G+D_H} \textup{ and } \lambda'=\frac{D_G+pD_H}{D_G+D_H}.
\end{align}
The proof includes the four following steps:
\begin{itemize}
\item[$(1)$] Show that
$\displaystyle W_1(\mu^p_{(x_1,y_1)},\mu^p_{(x_2,y_2)}) 
\ge W_1(\mu_{x_1}^{\lambda},\mu_{x_2}^{\lambda})+ W_1(\mu_{y_1}^{\lambda'},\mu_{y_2}^{\lambda'}).$
\item[$(2)$] Show that
\begin{align*}
W_1(\mu^p_{(x_1,y_1)},\mu^p_{(x_2,y_2)}) 
&\le \lambda'W_1(\mu_{x_1}^{p/{\lambda'}},\mu_{x_2}^{p/{\lambda'}})+(1-\lambda')d(x_1,x_2)+\\ & \ \ \ \ \lambda W_1(\mu_{y_1}^{p/{\lambda}},\mu_{y_2}^{p/{\lambda}})+(1-\lambda)d(y_1,y_2)
\end{align*}
\item[$(3)$] Show that the lower bound and upper bound of $\displaystyle W_1(\mu^p_{(x_1,y_1)},\mu^p_{(x_2,y_2)})$ given in $(1)$ and $(2)$ coincides for large enough $p\in[0,1]$. Hence, the inequality in Step (1) is indeed an equality (for large $p$).
\item[$(4)$] Derive the Lin-Lu-Yau curvature on the Cartesian product.
\end{itemize}

\underline{Step (1)} By Kantorovich Duality,
\begin{align*}
W_1(\mu^p_{(x_1,y_1)},\mu^p_{(x_2,y_2)})=\sup\limits_{\substack{\phi:V_{G\times H} \rightarrow\mathbb{R} \\ \phi\in \textup{1-Lip}}} \sum\limits_{\substack{w\in V_G \\ z\in V_H}} \phi((w,z)) \Big(\mu^p_{(x_1,y_1)}((w,z))-\mu^p_{(x_2,y_2)}((w,z))\Big)
\end{align*}

For each idleness $p\in [0,1]$, let $\Phi^{p}_G$ and $\Phi^{p}_H$ be optimal Kantorovich potentials transporting $\mu_{x_1}^p$ to $\mu_{x_2}^p$, and $\mu_{y_1}^p$ to $\mu_{y_2}^p$, respectively.
By Lemma \ref{lemma:cart_lip}, the function $\Phi((w,z)):=\Phi^{\lambda}_G(w)+\Phi^{\lambda'}_H(z)$ is 1-Lipschitz on $G\times H$, so it follows that
\begin{align} \label{eqn:sum_cart}
W_1(\mu^p_{(x_1,y_1)},\mu^p_{(x_2,y_2)})
&\ge \sum\limits_{(w,z)} \Big(\Phi^{\lambda}_G(w)+\Phi^{\lambda'}_H(z)\Big)\Big(\mu^p_{(x_1,y_1)}((w,z))-\mu^p_{(x_2,y_2)}((w,z))\Big).
\end{align}

The idea is to decompose a measure in cartesian product into a sum of measures in its coordinates. Consider characteristic equations of $\mu_{x_1}^{\lambda}$ and $\mu_{y_1}^{\lambda'}$ :
\begin{align*}
\mu_{x_1}^{\lambda}(w)=\lambda \mathbbm{1}_{x_1}(w) + \frac{(1-\lambda)}{D_G}\mathbbm{1}_{N(x_1)}(w),
\end{align*}
or equivalently
\begin{align} \label{eqn:char_lambda}
\mathbbm{1}_{N(x_1)}(w)=\frac{D_G}{1-\lambda}\Big(\mu_{x_1}^{\lambda}(w)-\lambda\mathbbm{1}_{x_1}(w)\Big).
\end{align}
Similarly, 
\begin{align} \label{eqn:char_lambda'}
\mathbbm{1}_{N(y_1)}(z)=\frac{D_H}{1-\lambda'}\Big(\mu_{y_1}^{\lambda'}(z)-\lambda'\mathbbm{1}_{y_1}(z)\Big).
\end{align}

Substitute \eqref{eqn:char_lambda} and \eqref{eqn:char_lambda'} into the characteristic equation of $\mu^p_{(x_1,y_1)}((w,z))$:
\begin{align*}
\mu^p_{(x_1,y_1)}((w,z)) &= p\mathbbm{1}_{x_1}(w)\mathbbm{1}_{y_1}(z)+\frac{1-p}{D_g+D_H}\Big(\mathbbm{1}_{N(x_1)}(w)\mathbbm{1}_{y_1}(z)+\mathbbm{1}_{x_1}(w)\mathbbm{1}_{N(y_1)}(z)\Big)\\
&=p\mathbbm{1}_{x_1}(w)\mathbbm{1}_{y_1}(z)+
K\Big(\mu_{x_1}^{\lambda}(w)-\lambda\mathbbm{1}_{x_1}(w)\Big)\mathbbm{1}_{y_1}(z) + \\
&\ \ \ \ K'\Big(\mu_{y_1}^{\lambda'}(z)-\lambda'\mathbbm{1}_{y_1}(z)\Big)\mathbbm{1}_{x_1}(w)\\
&=(p-K\lambda-K'\lambda')\mathbbm{1}_{x_1}(w)\mathbbm{1}_{y_1}(z)+K\mu_{x_1}^{\lambda}(w)\mathbbm{1}_{y_1}(z)+K'\mu_{y_1}^{\lambda'}(z)\mathbbm{1}_{x_1}(w)
\end{align*}
where constants $K=\frac{(1-p)D_G}{(1-\lambda)(D_G+D_H)}$ and $K'=\frac{(1-p)D_H}{(1-\lambda')(D_G+D_H)}$.

It follows that
\begin{align*}
\sum\limits_{(w,z)}\Phi^{\lambda}_G(w)\mu^p_{(x_1,y_1)}((w,z))
&=(p-K\lambda-K'\lambda')\Phi^{\lambda}_G(x_1)+K\sum\limits_{w}\Phi^{\lambda}_G(w)\mu_{x_1}^{\lambda}(w)+K'\Phi^{\lambda}_G(x_1)\\
&=(p-K\lambda-K'\lambda'+K')\Phi^{\lambda}_G(x_1)+K\sum\limits_{w}\Phi^{\lambda}_G(w)\mu_{x_1}^{\lambda}(w)
\end{align*}

With particular choice of $\lambda,\lambda'$ as in \eqref{eqn:choose_idleness}, we have $K=K'=1$ and $p-K\lambda-K'\lambda'+K'=0$. The equation above simply turns into
\begin{align*}
\sum\limits_{(w,z)}\Phi^{\lambda}_G(w)\mu^p_{(x_1,y_1)}((w,z))
=\sum\limits_{w}\Phi^{\lambda}_G(w)\mu_{x_1}^{\lambda}(w)
\end{align*}
Similarly, the same equation holds when subindex $1$ is replaced by $2$ everywhere, and therefore
\begin{align} \label{eqn:phi_G}
\sum\limits_{(w,z)}\Phi^{\lambda}_G(w)\Big(\mu^p_{(x_1,y_1)}((w,z))-\mu^p_{(x_2,y_2)}((w,z))\Big)
&=\sum\limits_{w}\Phi^{\lambda}_G(w)\mu_{x_1}^{\lambda}(w)-\sum\limits_{w}\Phi^{\lambda}_G(w)\mu_{x_2}^{\lambda}(w)  \nonumber \\
&=W_1(\mu_{x_1}^{\lambda},\mu_{x_2}^{\lambda})
\end{align}

By similar arguments, we also have
\begin{align} \label{eqn:phi_H}
\sum\limits_{(w,z)}\Phi^{\lambda}_H(z)\Big(\mu^p_{(x_1,y_1)}((w,z))-\mu^p_{(x_2,y_2)}((w,z))\Big)=W_1(\mu_{y_1}^{\lambda'},\mu_{y_2}^{\lambda'})
\end{align}

Combining \eqref{eqn:phi_G} and \eqref{eqn:phi_H} into \eqref{eqn:sum_cart}, we obtain
\begin{align*}
W_1(\mu^p_{(x_1,y_1)},\mu^p_{(x_2,y_2)}) 
&\ge W_1(\mu_{x_1}^{\lambda},\mu_{x_2}^{\lambda})+ W_1(\mu_{y_1}^{\lambda'},\mu_{y_2}^{\lambda'}).
\end{align*}

\underline{Step (2)}
By metric property of $W_1$,
\begin{align} \label{eqn:metric_cart}
W_1(\mu^p_{(x_1,y_1)},\mu^p_{(x_2,y_2)}) \le W_1(\mu^p_{(x_1,y_1)},\mu^p_{(x_2,y_1)})+ W_1(\mu^p_{(x_2,y_1)},\mu^p_{(x_2,y_2)})
\end{align}

For each $i\in \{1,2\}$, the measure $\mu^{p/{\lambda'}}_{x_i}$ satisfies
\begin{align*}
\lambda'\cdot \mu^{p/{\lambda'}}_{x_i} (w) = 
\begin{cases}
p &\textup{ if } w=x_i\\
\lambda'\left(\frac{1-\frac{p}{\lambda'}}{D_G}\right)=\frac{1-p}{D_G+D_H} &\textup{ if } w\sim x_i\\
0 &\textup{ otherwise.}
\end{cases}
\end{align*}

Let $\pi_x\in\prod(\mu^{p/{\lambda'}}_{x_1},\mu^{p/{\lambda'}}_{x_2})$ be an optimal transport plan from $\mu^{p/{\lambda'}}_{x_1}$ to $\mu^{p/{\lambda'}}_{x_2}$.
Consider a transport plan $\pi\in \prod(\mu^p_{(x_1,y_1)},\mu^p_{(x_2,y_1)})$ given by 
\begin{align*}
\pi \big((x_1,z),(x_2,z) \big) &= \frac{1-p}{D_G+D_H} \ \ \ \ \textup{ for all } z\stackrel{H}{\sim} y\\
\pi \big((w_1, y_1),(w_2,y_1)\big)&=\lambda'\pi_x(w_1,w_2) \ \textup{ for all } w_1,w_2\in V_G \\
\pi \big((w_1, z_1),(w_2,z_2)\big)&=0 \ \ \ \ \ \ \ \ \ \ \ \ \ \ \ \ \textup{ everywhere else.}
\end{align*}

The cost of plan $\pi$ then give an upper bound for $W_1(\mu^p_{(x_1,y_1)},\mu^p_{(x_2,y_1)})$:
\begin{align} \label{eqn:plan_pi_x}
W_1(\mu^p_{(x_1,y_1)},\mu^p_{(x_2,y_1)})
&\le \sum\limits_{(w_1,z_1),(w_2,z_2)} \pi\big((w_1,z_1),(w_2,z_2)\big)d \big((w_1,z_1),(w_2,z_2)\big)\nonumber \\
&=\frac{(1-p)D_H}{D_G+D_H}d(x_1,x_2)+\lambda'W_1(\mu^{p/{\lambda'}}_{x_1},\mu^{p/{\lambda'}}_{x_2}) \nonumber \\
&=(1-\lambda')d(x_1,x_2)+\lambda'W_1(\mu^{p/{\lambda'}}_{x_1},\mu^{p/{\lambda'}}_{x_2}).
\end{align}

By similar arguments, we can derive
\begin{align} \label{eqn:plan_pi_y}
W_1(\mu^p_{(x_1,y_1)},\mu^p_{(x_2,y_1)})
\le (1-\lambda)d(y_1,y_2)+\lambda W_1(\mu^{p/{\lambda}}_{y_1},\mu^{p/{\lambda}}_{y_2}).
\end{align}

Combining \eqref{eqn:plan_pi_x} and \eqref{eqn:plan_pi_y} into \eqref{eqn:metric_cart} results in
\begin{align*}
W_1(\mu^p_{(x_1,y_1)},\mu^p_{(x_2,y_2)}) 
&\le \lambda'W_1(\mu_{x_1}^{p/{\lambda'}},\mu_{x_2}^{p/{\lambda'}})+(1-\lambda')d(x_1,x_2)+\\ & \ \ \ \ \lambda W_1(\mu_{y_1}^{p/{\lambda}},\mu_{y_2}^{p/{\lambda}})+(1-\lambda)d(y_1,y_2).
\end{align*}

\underline{Step (3)} The equation 
\begin{align} \label{eqn:compare_lambda_x}
W_1(\mu_{x_1}^{\lambda},\mu_{x_2}^{\lambda})=\lambda'W_1(\mu_{x_1}^{p/{\lambda'}},\mu_{x_2}^{p/{\lambda'}})+(1-\lambda')d(x_1,x_2)
\end{align}
does not hold true in general. However, we will show that it is true for all $p\in[0,1]$ large enough. Recall that $$\lambda=\frac{pD_G+D_H}{D_G+D_H} \textup{ and } \frac{p}{\lambda'}=\frac{pD_G+pD_H}{D_G+pD_H}$$ which are both increasing functions of $p\in[0,1]$ and reach value 1 when $p=1$. We consider $p$ large enough so that $\lambda,\frac{p}{\lambda'} \ge \frac{1}{2}$.
Recall Lemma \ref{lemma:intersectionpt} that for all idleness $q\in [\frac{1}{2},1]$, $$W_1(\mu_{x_1}^q, \mu_{x_2}^q) =f_{\delta}(q):=q\cdot \delta+(1-q)c_{\delta}$$ where $\delta=d(x_1,x_2).$
In other words, $$\frac{W_1(\mu_{x_1}^q, \mu_{x_2}^q)-q\cdot \delta}{1-q} = c_{\delta}=\frac{W_1(\mu_{x_1}^r, \mu_{x_2}^r)-r\cdot \delta}{1-r}$$ for all idleness $q,r\in [\frac{1}{2},1]$.

Substitution $q=\lambda$ and $r=\frac{p}{\lambda'}$ gives
\begin{align*}
W_1(\mu_{x_1}^{\lambda}, \mu_{x_2}^{\lambda}) 
&=\lambda\cdot \delta + (1-\lambda)\bigg(\frac{W_1(\mu_{x_1}^{p/\lambda'}, \mu_{x_2}^{p/\lambda'})-\frac{p}{\lambda'}\cdot \delta}{1-\frac{p}{\lambda'}}\bigg)\\
&=\lambda'W_1(\mu_{x_1}^{p/\lambda'}, \mu_{x_2}^{p/\lambda'})+(1-\lambda')\delta
\end{align*}
where the second line using the identity $1-\lambda=\lambda'-p$, yielding the equation \eqref{eqn:compare_lambda_x}.

Similarly, for large $p$ such that $\lambda',\frac{p}{\lambda}\ge \frac{1}{2}$, we also have
\begin{align} \label{eqn:compare_lambda_y}
W_1(\mu_{y_1}^{\lambda'},\mu_{y_2}^{\lambda'})=\lambda W_1(\mu_{y_1}^{p/{\lambda}},\mu_{y_2}^{p/{\lambda}})+(1-\lambda)d(y_1,y_2)
\end{align}

We can then conclude that, when $p\in[0,1]$ is large enough (so that $p,\lambda,\lambda',\frac{p}{\lambda}, \frac{p}{\lambda'} \ge \frac{1}{2}$), the lower bound and upper bound of $W_1(\mu^p_{(x_1,y_1)},\mu^p_{(x_2,y_2)})$ agree. 

\underline{Step (4)}
The previous steps implies that, for large enough $p\in[0,1]$, $$W_1(\mu^p_{(x_1,y_1)},\mu^p_{(x_2,y_2)})=W_1(\mu_{x_1}^{\lambda},\mu_{x_2}^{\lambda})+ W_1(\mu_{y_1}^{\lambda'},\mu_{y_2}^{\lambda'}).$$

Finally, we can translate this relation in term of Lin-Lu-Yau curvature, using Corollary \ref{cor:LLY}: for any $a_1\not=a_2\in A$ and any $q\in [\frac{1}{2},1]$,
$$\kappa_{LLY}^A(a_1,a_2)\stackrel{\textup{Cor } \ref{cor:LLY}}{=}\frac{\kappa_q(a_1,a_2)}{1-q}= \frac{1}{1-q}\Big(1-\frac{W_1(\mu_{a_1}^{q}, \mu_{a_1}^{q})}{d(a_1,a_2)}\Big).$$

For abbreviation, $\kappa_{LLY}^{G\times H}=\kappa_{LLY}^{G\times H}((x_1,y_1),(x_2,y_2))$ and  $\kappa_{LLY}^{G}=\kappa_{LLY}^{G}(x_1,x_2)$ and $\kappa_{LLY}^{H}=\kappa_{LLY}^{H}(y_1,y_2)$
\begin{align*}
\kappa_{LLY}^{G\times H}&=\frac{1}{1-p}\Big(1-\frac{W_1(\mu_{x_1}^{\lambda},\mu_{x_2}^{\lambda})+ W_1(\mu_{y_1}^{\lambda'},\mu_{y_2}^{\lambda'})}{d(x_1,x_2)+d(y_1,y_2)}\Big) \\
&=\frac{1}{1-p}\Big(1-\frac{d(x_1,x_2)\cdot(1-(1-\lambda)\kappa_{LLY}^G)+ d(y_1,y_2)\cdot(1-(1-\lambda')\kappa_{LLY}^H)}{d(x_1,x_2)+d(y_1,y_2)}\Big)\\
&=\frac{d(x_{1},x_{2})(1-\lambda)\kappa^{G}_{LLY}+d(y_{1},y_{2})(1-\lambda')\kappa^{H}_{LLY}}{(1-p)(d(x_{1},x_{2})+d(y_{1},y_{2}))}\\
&=\frac{d(x_{1},x_{2})D_{G}\kappa^{G}_{LLY}+d(y_{1},y_{2})D_{H}\kappa^{H}_{LLY}}{(D_{G}+D_{H})(d(x_{1},x_{2})+d(y_{1},y_{2}))}.
\end{align*}
Furthermore, for all $p\in [\frac{1}{2},1)$, Corollary \ref{cor:LLY}: $\kappa_p=(1-p)\kappa_{LLY}$ allows us to replace $\kappa_{LLY}$ by $\kappa_p$ in the above equation, which completes the proof. 
\end{proof}

\section{Long scale behaviour} \label{sect:  behavior}

In most papers regarding Ollivier-Ricci curvature, only the short scale curvature is usually concerned because the curvature given at an edge $x\sim y$ is a discrete analogue to the Ricci curvature given at a unit tangent vector. Moreover, a lower bound on the short scale curvature implies the same lower bound for the curvature between any two points (see \cite[Proposition]{Oll}):
$$\textup{If } \kappa_p(x,y)\ge \kappa \textup{ for all }x\sim y, \textup{then } \kappa_p(x,y)\ge \kappa \textup{ for all }x,y\in V.$$
However, restricting oneself only to short scale curvature could lead to some contradiction to the nature of particular graphs, e.g. the hexagonal tiling as illustrated in the following subsection. Later, we then discuss about some global implication of curvature signs.

\subsection{The hexagonal tiling} \label{subs: hexagon}
Let $G=(V,E)$ be a graph of the hexagonal tiling (which maybe either infinite tessellation, or finite tessellation, e.g. on a torus $T^2$).

Consider a pair of points $(x,y)$ with distance $d(x,y)=7.$ There are 4 non-equivalent positions of $y$ relative to $x$ listed as $y_1, y_2,y_3,y_4$ as shown in Figure \ref{fig:hexagon_tiling}.

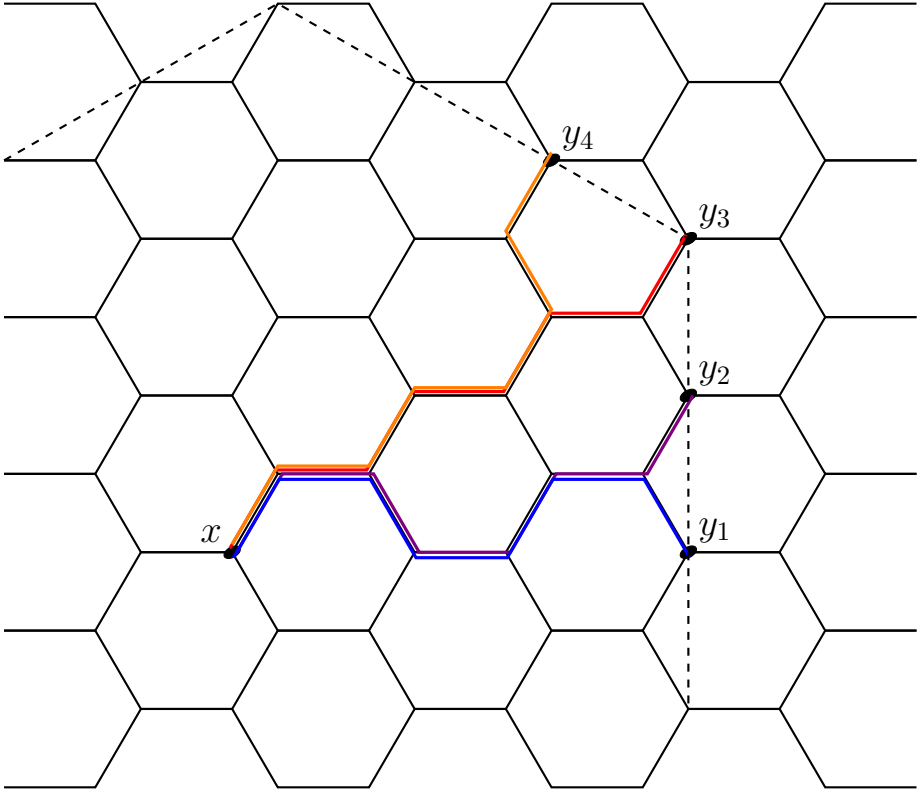
\begin{figure}[h!]
	\begin{tikzpicture}[scale=1.2]
	\pgftransformcm{1}{0}{0.500}{0.866}{\pgfpoint{0}{0}}
	\foreach \x in {0,1,2,3,4} {
		\drawHEX{-1*\x, 2*\x};
	}
	\drawlabelnode{-2,3}{above left}{$x$};
	\drawlabelnode{1,7}{above right}{$y_3$};
	\drawlabelnode{2,5}{above right}{$y_2$};
	\drawlabelnode{3,3}{above right}{$y_1$};
	\drawlabelnode{-1,8}{above right}{$y_4$};
	
	\draw[red,very thick,shift={(-0.05,0.05)}] 
	(-2,3)--(-2,4)--(-1,4)--(-1,5)--(0,5)--(0,6)--(1,6)--(1,7);
	
	\draw[violet,very thick,shift={(0.05,0)}] 
	(-2,3)--(-2,4)--(-1,4)--(0,3)--(1,3)--(1,4)--(2,4)--(2,5);
	
	\draw[blue,very thick,shift={(0.05,-0.07)}] 
	(-2,3)--(-2,4)--(-1,4)--(0,3)--(1,3)--(1,4)--(2,4)--(3,3);
	
	\draw[dashed,thick,shift={(0,0)}] 
	(-7,8)--(-5,10)--(1,7)--(4,1);
	
	\draw[orange,very thick,shift={(-0.05,0.1)}] 
	(-2,3)--(-2,4)--(-1,4)--(-1,5)--(0,5)--(0,6)--(-1,7)--(-1,8);	
	
	\end{tikzpicture}
	\caption{The dashed lines represent the locus of points with distance $7$ from $x$ in the hexagonal tiling.}
	\label{fig:hexagon_tiling}
\end{figure}

The following proposition gives the formula of short scale and long scale curvature (of distance 7) of the hexagonal tiling. The proof is omitted.
\begin{proposition} \label{prop: hexagon_curvature}
	Let $G=(V,E)$ be a graph of the hexagonal tiling. Let $p\in [0,1]$ be an idleness parameter. Then
	\begin{itemize}
		\item[\textup(i)] for any $w,z\in V$ such that $w\sim z$, the curvature is given by $$\kappa_p(w,z)=-\frac{2}{3}(1-p)<0.$$
		
		\item[\textup(ii)] for $x,y\in V$ such that $d(x,y)=7$, the (long scale) curvature is given by $$\kappa_p(x,y_1)=\kappa_p(x,y_2)=\frac{2}{21}(1-p)>0$$ and
		$$\kappa_p(x,y_3)=\kappa_p(x,y_4)=-\frac{2}{21}(1-p)<0.$$
	\end{itemize}
\end{proposition}

In particular, consider a finite hexagonal tessellation on a torus $T^2$, where the space is expected to have both positive and negative curvature. However, in short scale curvature, only the negative sign is presented, which suggests that the long scale curvature is more suitable to describe this space.

\subsection{Global results}

\begin{theorem}[non-positive curvature] \label{thm: neg_curvature}
	Let $G=(V,E)$ be locally finite graph and let $p\in[0,1)$. Assume that the curvature $\kappa_p(x,y)\le 0$ for all $x\not=y\in V$. Then $G$ must be infinite.
\end{theorem}

\begin{proof}
	Suppose for the sake of contradiction that $G$ is finite, with the diameter $$\textup{diam}(G):=\sup\{d(w,z): w,z\in V\}=L<\infty.$$
	Note also that $L\le 2$, because $G$ cannot be a complete graph, which has positive curvature.
	Let $x$ and $y$ be antipodal vertices in $V$ (i.e. $d(x,y)=L$). Consider a geodesic from $x$ to $y$, namely $x=v_0\sim v_1 \sim...\sim v_{L-1}\sim v_L=y$. It follows that $v_1$ is a neighbour of $x$, and $v_{L-1}$ is a neighbour of $y$, and that $d(v_1,v_{L-1})=L-2$. Consider a transport plan $\pi\in \prod(\mu_x^p,\mu_y^p)$ such that $\pi(v_1,v_{L-1})>0$.
	
	Hence the $W_1(\mu_x^p,\mu_y^p)$ is bounded above by:
	\begin{align} \label{eqn: neg_curvature}
	W_1(\mu_x^p,\mu_y^p)\le \sum\limits_{w,z\in V} \pi(w,z)d(w,z) \le L\cdot\sum\limits_{w,z\in V} \pi(w,z)=L.
	\end{align}
	
	Moreover, $\pi(w,z)d(w,z) < L\pi(w,z)$ when $w=v_1$ and $z=v_{L-1}$, so the inequality in \eqref{eqn: neg_curvature} must be strict. That is $W_1(\mu_x^p,\mu_y^p)< L$ which then implies $$\kappa_p(x,y)=1-\frac{1}{L}W_1(\mu_x^p,\mu_y^p)>0$$ contradicting to the curvature assumption.
\end{proof}

\begin{example}
Let $G=(V,E)$ be the infinite $d$-regular tree (with $d\ge 0$). Let $x,y\in V$ such that $d(x,y)=L\ge 1$. Then for all $p\in [0,1]$, $$\kappa_p(x,y)=\frac{4-2d}{dL}(1-p).$$ The infinite regular trees illustrate a family of graphs which have non-positive curvature everywhere.
\end{example}

\begin{remark}
	The thoerem of discrete Bonnet-Myers \cite{LLY11,Oll} states that a graph $G=(V,E)$ with positive curvature bounded away from zero $\kappa_p(x,y)\ge K > 0$ for all $x\not=y\in V$ must be a finite graph. This assumption can be replaced by: $\kappa_p(x,y)\ge K > 0$ for all neighbours $x\sim y\in V$, since both assumptions are essentially equivalent. On the other hand, the assumption in Theorem \ref{thm: neg_curvature} cannot be reduced to: $\kappa_p(x,y)\le 0$ for all neighbours $x\sim y\in V$. As a counterexample, consider a graph $G$ of a finite hexagonal tessellation on a torus $T^2$ (see Subsection \ref{subs: hexagon}).
\end{remark}

There is another way to modify discrete Bonnet-Myers' theorem, by replacing the assumption condition with $\kappa_p(x,y)\ge \kappa> 0$ for a fixed vertex $x\in V$ and for all $y\in V$.

\begin{theorem}[modified discrete Bonnet-Myers] \label{thm: pos_curvature}
	Let $G=(V,E)$ be locally finite graph and let $p\in[0,1)$. Assume that there is a constant $\kappa>0$ and a fixed vertex $x\in V$ such that the curvature $\kappa_p(x,y)\ge \kappa$ for all $y\in V\backslash\{x\}$. Then $G$ must be finite.
\end{theorem}

The proof is very similar to the one in the original discrete Bonnet-Myers \cite{Oll}, which employs the dirac-measure $\delta_x=\mathbbm{1}_x$.

\begin{proof}
Consider any $y\in V$. Let $L:=d(x,y)$. The condition $\kappa_p(x,y)\ge \kappa$ implies that $W_1(\mu_x^p,\mu_y^p) \le (1-\kappa)L$. 
Then 
\begin{align*}
L=W_1(\delta_x,\delta_y) &\le W_1(\delta_x,\mu_x^p)+W_1(\mu_x^p,\mu_y^p)+W_1(\mu_y^p,\delta_y)\\
&=W_1(\mu_x^p,\mu_y^p)+2(1-p)\\
&\le (1-\kappa)L+2(1-p)
\end{align*}
which gives $d(x,y)\le \frac{2(1-p)}{\kappa}$ for all $y\in V$. Hence, $d(y,z)\le d(x,y)+d(x,z)\le \frac{4(1-p)}{\kappa}$ for all $y,z\in V$.
\end{proof}

\begin{example}
Let $G=(V,E)$ be the graph as shown in Figure \ref{fig:example_BM}. The short scale Lin-Lu-Yau curvature of $G$ can be computed as:
\begin{align*}
\kappa_{LLY}(x,w)=1; \ \kappa_{LLY}(w,y)=-\frac{1}{3}; \
\kappa_{LLY}(y,z_1)=\kappa_{LLY}(y,z_2)=\frac{2}{3}. 
\end{align*}
Therefore $G$ does not satisfy the condition of the original discrete Bonnet-Myers' theorem. However, the curvature between $x$ and the other vertices are as follows:
\begin{align*}
\kappa_{LLY}(x,w)=1; \ \kappa_{LLY}(x,y)=\frac{1}{3}; \
\kappa_{LLY}(x,z_1)=\kappa_{LLY}(x,z_2)=\frac{2}{3}. 
\end{align*}
Thus $G$ does however satisfy the condition of Theorem \ref{thm: pos_curvature}.

\end{example}

\begin{figure}[h!]
	\centering
	\begin{tikzpicture}[scale=2]
	\drawlabelnode{-2,0}{above left}{$x$};
	\drawlabelnode{-1,0}{above left}{$w$};
	\drawlabelnode{0,0}{above left}{$y$};
	\drawlabelnode{0.5,0.866}{above right}{$z_1$};
	\drawlabelnode{0.5,-0.866}{above right}{$z_2$};

	\draw[black, very thick] (-2,0)--(-1,0)--(0,0)--(0.5,0.866);
	\draw[black, very thick] (0,0)--(0.5,-0.866);
	\end{tikzpicture}
	\caption{The graph $G$}
	\label{fig:example_BM}
\end{figure}
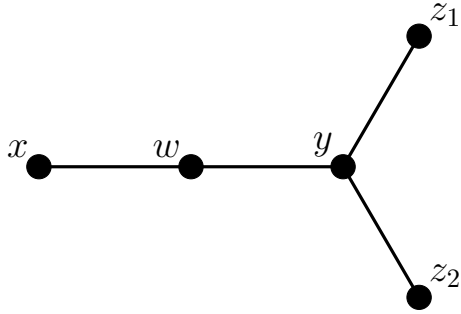

Lastly, we finish the paper by quoting an open problem on long scale curvature suggested by Ollivier in \cite{Ollopen}.

\begin{problem} (\cite[Problem C]{Ollopen})
Ricci curvature of $\mathbb{Z}^n$ is 0. What happens on discrete or continuous nilpotent groups?
For example, on the discrete Heisenberg group $$H_3(\mathbb{Z}):=\langle a,b,c \ \big| ac=ca,\ bc=cb,\ [a,b]=c\rangle,$$ the natural discrete random walk analogous to the hypoelliptic diffusion operator on the continuous Heisenberg group is the random walk generated by $a$ and $b$. Since these
generators are free up to length 8, clearly Ricci curvature is negative at small scales, but
does it tend to 0 at larger and larger scales?
\end{problem}

\end{document}